\documentclass{amsart}
\calclayout
 \usepackage{amssymb,amsmath,amsfonts,epsfig,latexsym,tikz}
 \usepackage[alphabetic]{amsrefs}
 \usepackage{tikz-cd}
\usepackage{hyperref}
\usepackage{enumerate}
\usepackage{mathtools}
\usepackage{verbatim}
\usepackage{cleveref}
\usepackage[shortlabels]{enumitem}
\usepackage{subcaption}

\usetikzlibrary{positioning}
\usetikzlibrary{matrix}
\usetikzlibrary{decorations}
\usetikzlibrary{decorations.pathreplacing, decorations.pathmorphing, angles,quotes}
 
 \newtheorem{theorem}{Theorem}[section]
\newtheorem{proposition}[theorem]{Proposition}
\newtheorem{lemma}[theorem]{Lemma}
\newtheorem{corollary}[theorem]{Corollary}

\theoremstyle{definition}
\newtheorem{definition}[theorem]{Definition}
\newtheorem{remark}[theorem]{Remark}
\newtheorem{example}[theorem]{Example}

\begin{document}

\title{Straightening laws for Chow rings of matroids}          
    
\author{Matt Larson}
\address{Stanford U. Department of Mathematics, 450 Jane Stanford Way, Stanford, CA 94305}
\email{mwlarson@stanford.edu}

\begin{abstract}
We give elementary and non-inductive proofs of three fundamental theorems about Chow rings of matroids: the standard monomial basis, Poincar\'e duality, and the dragon-Hall--Rado formula. Our approach, which also works for augmented Chow rings of matroids, is based on a straightening law. This approach gives a decomposition of the Chow ring of a matroid into pieces indexed by flats. 
\end{abstract}
 
\maketitle

\vspace{-20 pt}

\section{Introduction}

A \emph{matroid} $\mathrm{M}$ is a finite nonempty atomic ranked lattice $\mathcal{L}_{\mathrm{M}}$ whose rank function $\operatorname{rk} \colon \mathcal{L}_{\mathrm{M}} \to \mathbb{Z}$ is submodular:
$$\operatorname{rk}(F_1 \vee F_2) + \operatorname{rk}(F_1 \wedge F_2) \le \operatorname{rk}(F_1) + \operatorname{rk}(F_2) \text{ for all }F_1, F_2 \in \mathcal{L}_{\mathrm{M}}.$$
The minimal element of $\mathcal{L}_{\mathrm{M}}$ is usually denoted $\emptyset$ and the maximal element is the \emph{ground set}, which is usually denoted $E$. 
That $\mathcal{L}$ is \emph{atomic} means that every element is the join of the atoms it contains, and that it is \emph{ranked} means that every maximal chain in an interval $[\emptyset, F]$ has the same length, which is $\operatorname{rk}(F)$. 
The \emph{rank} of a matroid is $\operatorname{rk}(E)$. The elements of $\mathcal{L}_{\mathrm{M}}$ are called \emph{flats}. Let $\overline{\mathcal{L}}_{\mathrm{M}} = \mathcal{L}_{\mathrm{M}} \setminus \{\emptyset\}$. 

\begin{definition}
The \emph{Chow ring} $\underline{A}^{\bullet}(\mathrm{M})$ of $\mathrm{M}$ is the ring given by the presentation
$$\underline{A}^{\bullet}(\mathrm{M}) = \frac{\mathbb{Z}[h_F]_{F \in \overline{\mathcal{L}}_{\mathrm{M}}}}{((h_{F} - h_{G \vee F})(h_G - h_{G \vee F}) : F, G \in \overline{\mathcal{L}}_{\mathrm{M}}) + (h_a:  \text{ }a \text{ atom})}.$$
\end{definition}

Chow rings of matroids were first considered in \cite{FY} as a generalization of Chow rings of the wonderful compactifications of hyperplane arrangement complements, which were introduced in \cite{dCP95}. They play a key role in the proof of log-concavity results for matroids \cites{AHK18,BST20,ADH}.
The above definition is called the simplicial presentation of $\underline{A}^{\bullet}(\mathrm{M})$. It was first considered in \cite{YuzvinskySimplicial} and then extensively studied in \cite{BES}. 
See \cite{LLPP}*{Appendix A} for a proof of the equivalence between the above definition of $\underline{A}^{\bullet}(\mathrm{M})$ and the definition used in \cite{FY}. The Chow ring of a matroid is graded, with each $h_F$ in degree $1$. We now state three fundamental results about Chow rings of matroids. 

\begin{theorem}\cites{FY,BES,AHK18}\label{thm:nonaug}
Let $\mathrm{M}$ be a matroid of rank $r$. Then
\begin{enumerate}
\item The monomials
\begin{equation}\tag{\underline{SM}} \label{SMnonaug}
\{h_{F_1}^{a_1} \dotsb h_{F_\ell}^{a_\ell} : \emptyset = F_0 < F_1 < \dotsb < F_\ell, \text{ } a_i < \operatorname{rk}(F_i) - \operatorname{rk}(F_{i-1}) \text{ for }i=1, \dotsc, \ell\}
\end{equation}
form an integral basis for $\underline{A}^{\bullet}(\mathrm{M})$. 
\item  There is an isomorphism ${\operatorname{deg}} \colon \underline{A}^{r-1}(\mathrm{M}) \to \mathbb{Z}$ given by
\begin{equation}\tag{\underline{dHR}}\label{dHR}
{\operatorname{deg}}(h_{F_1} \dotsb h_{F_{r-1}}) = \begin{cases} 1 & \text{ if for all }\emptyset \not= T \subseteq [r-1], \text{ }\operatorname{rk}(\bigvee_{i \in T} F_i) \ge |T| + 1, \\ 0 & \text{ otherwise}.\end{cases}\end{equation}
\item The pairing 
\begin{equation}\tag{\underline{PD}} \label{PDnonaug}
\underline{A}^k(\mathrm{M}) \times \underline{A}^{r-1-k}(\mathrm{M}) \to \mathbb{Z} \text{ given by } (a, b) \mapsto \operatorname{deg}(ab)
\end{equation} is unimodular, i.e., it defines an isomorphism $\underline{A}^k(\mathrm{M}) \to \operatorname{Hom}(\underline{A}^{r-1-k}(\mathrm{M}), \mathbb{Z})$.
\end{enumerate}
\end{theorem}

The augmented Chow ring of a matroid is a variant of the Chow ring of a matroid introduced in \cite{BHMPW20a}. It plays a key role in the proof of the top-heavy conjecture in \cite{BHMPW20b}.

\begin{definition}
The \emph{augmented Chow ring} $A^{\bullet}(\mathrm{M})$ of $\mathrm{M}$ is the ring given by the presentation
$$A^{\bullet}(\mathrm{M}) = \frac{\mathbb{Z}[h_F]_{F \in \overline{\mathcal{L}}_{\mathrm{M}}}}{((h_{F} - h_{G \vee F})(h_G - h_{G \vee F}) : F, G \in \overline{\mathcal{L}}_{\mathrm{M}}) + (h_a^2, h_ah_F - h_ah_{F \vee a} : F \in \overline{\mathcal{L}}_{\mathrm{M}} , \text{ }a \text{ atom})}.$$
\end{definition}
See \cite{LLPP}*{Appendix A} for a proof of the equivalence between the above definition and the definition used in \cite{BHMPW20a}. Note that $\underline{A}^{\bullet}(\mathrm{M})$ is a quotient of $A^{\bullet}(\mathrm{M})$. We now state three fundamental results about augmented Chow rings of matroids. 

\begin{theorem}\cites{ELPoly,BHMPW20a}\label{thm:aug}
Let $\mathrm{M}$ be a matroid of rank $r$. Then
\begin{enumerate}
\item The monomials
\begin{equation}\tag{SM} \label{SM}
\{h_{F_1}^{a_1} \dotsb h_{F_\ell}^{a_\ell} : \emptyset = F_0 < F_1 < \dotsb < F_\ell, \text{ } a_1 \le \operatorname{rk}(F_1), \text{ }  a_i < \operatorname{rk}(F_i) - \operatorname{rk}(F_{i-1}) \text{ for }i=2, \dotsc, \ell\}
\end{equation}
form an integral basis for ${A}^{\bullet}(\mathrm{M})$. 

\item There is an isomorphism ${\operatorname{deg}} \colon {A}^{r}(\mathrm{M}) \to \mathbb{Z}$ given by
\begin{equation}\tag{HR}\label{HR}
{\operatorname{deg}}(h_{F_1} \dotsb h_{F_{r}}) = \begin{cases} 1 & \text{ if for all } T \subseteq [r], \text{ }\operatorname{rk}(\bigvee_{i \in T} F_i) \ge |T|, \\ 0 & \text{ otherwise}.\end{cases}
\end{equation}
\item The pairing
\begin{equation}\tag{{PD}} \label{PD}
{A}^k(\mathrm{M}) \times {A}^{r-k}(\mathrm{M}) \to \mathbb{Z} \text{ given by } (a, b) \mapsto \operatorname{deg}(ab)
\end{equation} is unimodular.
\end{enumerate}
\end{theorem}

We give elementary and non-inductive proofs of Theorems \ref{thm:nonaug} and \ref{thm:aug}. We use only the above definition of a matroid and basic linear algebra. 
We now discuss the history of the above results. 

Theorem \ref{SMnonaug} and \ref{SM} give \emph{standard monomial} bases for (augmented) Chow rings of matroids. A Gr\"{o}bner basis for $\underline{A}^{\bullet}(\mathrm{M})$ was given in \cite{FY}, and this gives a monomial basis for $\underline{A}^{\bullet}(\mathrm{M})$. In \cite{BES}*{Corollary 3.3.3}, it is shown that this monomial basis is essentially equivalent to the one given in Theorem \ref{SMnonaug}. Theorem \ref{SM} has not appeared explicitly in the literature before, but it is well-known to experts. 
Using the ``free coextension trick'', the result of \cite{FY} can be used to produce a Gr\"{o}bner basis for $A^{\bullet}(\mathrm{M})$ as well; see \cite{MatroidKoszul}*{Section 5}. After some further manipulations this yields Theorem \ref{SM}; see \cite{EHL}*{Theorem 7.7} for a special case. We note that Theorem~\ref{SMnonaug} can be easily used to prove that the Gr\"{o}bner basis given in \cite{FY} is indeed a Gr\"{o}bner basis. 

Theorem \ref{dHR} and \ref{HR} are known as the \emph{dragon Hall--Rado} and \emph{Hall--Rado} formula, respectively, after the Hall--Rado theorem in matroid theory \cite{Rado}. Theorem \ref{dHR} is a generalization of Postnikov's dragon marriage theorem \cite{Pos09}*{Theorem 9.3}, and it was proven in \cite{BES}*{Theorem 5.2.4} using an inductive argument based on \cite{SpeyerTropicalLinearSpace}*{Proposition 4.4}, which relies on a connection between $\underline A^{\bullet}(\mathrm{M})$ and the permutohedral toric variety. Theorem~\ref{HR} was proven in \cite{ELPoly}*{Theorem 1.3} using a polyhedral interpretation of $A^{\bullet}(\mathrm{M})$. The argument given there can be adapted to prove Theorem \ref{dHR}; see \cite{ELPoly}*{Remark 6.3}. See also \cite{EFLS}*{Corollary 4.8}. Even the existence of the isomorphism $\deg$, which is called the \emph{degree map}, is nontrivial. It can be constructed using a tropical interpretation of the  Chow ring, see \cite{AHK18}*{Definition 5.9}. 

Theorem \ref{PDnonaug} and \ref{PD} state that (augmented) Chow rings of matroids satisfy a version of \emph{Poincar\'{e} duality}. Theorem \ref{PDnonaug} was first proven in \cite{AHK18}*{Theorem 6.19} using an inductive argument. Different inductive proofs have been given in \cites{BHMPW20a,LerayModel}. Non-inductive arguments using Theorem \ref{SMnonaug} have been given in \cites{BES,DastidarRoss,PagariaPezzoli}. Theorem~\ref{PD} was proven in \cite{BHMPW20a}*{Theorem 1.3(4)} using an inductive argument. It can also be deduced from \cite{AHK18}*{Theorem 6.19}; see \cite{BHMPW20a}*{Remark 4.1}. 

There is a generalization of the Chow ring of a matroid to take into account a \emph{building set} on the lattice of flats. Yuzvinsky gave an analogue of Theorem~\ref{SMnonaug} and Theorem~\ref{PDnonaug} for Chow rings of realizable matroids at the \emph{minimal building set} \cite{YuzvinskyCohomologyBasis}. Yuzvinsky's argument requires significant effort to adapt it to all matroids. Feichtner and Yuzvinsky give a Gr\"{o}bner basis, and therefore a standard monomial basis, for the Chow ring of a matroid at any building set \cite{FY}. These Gr\"{o}bner basis arguments are further generalized in \cites{LerayModel,PagariaPezzoli}.

Besides Poincar\'{e} duality, (augmented) Chow rings of matroids satisfy the other parts of the \emph{K\"{a}hler package}: the Hard Lefschetz theorem and the Hodge--Riemann relations. At the moment, the only proofs of the full K\"{a}hler package rely on intricate inductions \cites{AHK18,BHMPW20a,PagariaPezzoli}.

\medskip
Our approach begins with the augmented Chow ring. We use a ``straightening'' procedure which allows us to rewrite any monomial in terms of the standard monomials. This implies that the standard monomials span $A^{\bullet}(\mathrm{M})$, and so $A^r(\mathrm{M})$ has dimension at most $1$. We then directly verify that $\operatorname{deg} \colon A^r(\mathrm{M}) \to \mathbb{Z}$ given in Theorem~\ref{HR} is well-defined and an isomorphism. Finally, we prove Poincar\'{e} duality and prove the linear independence of the standard monomials simultaneously by showing that a certain matrix is lower triangular. 
With some additional arguments, we can deduce Theorem~\ref{thm:nonaug} because $\underline A^{\bullet}(\mathrm{M})$ is a quotient of $A^{\bullet}(\mathrm{M})$. 

Our approach to Poincar\'{e} duality is closely related to the approach in \cite{BES}, which is in turn inspired by an argument of Hampe \cite{Ham17} in the case of Boolean matroids. However, there are significant differences. For example, the argument in \cite{BES} relies on Poincar\'{e} duality for Boolean matroids.

\medskip 
In Section~\ref{sec:2}, we prove Theorem~\ref{thm:aug} and then deduce Theorem~\ref{thm:nonaug} from it. In Section~\ref{sec:3}, we explain a consequences of our approach: the (augmented) Chow ring of a matroid has a direct sum decomposition indexed by $\overline{\mathcal{L}}_{\mathrm{M}}$. We use this to derive a new recursion for the Hilbert series of $A^{\bullet}(\mathrm{M})$ and $\underline{A}^{\bullet}(\mathrm{M})$. This decomposition also gives different proof of Theorem~\ref{SM}. In Section~\ref{ref:sec4}, we construct an \emph{algebra with straightening law} related to the Chow ring of a matroid. We use this to give another proof of Theorem~\ref{SMnonaug}. 

\subsection*{Notation} Throughout, $\mathrm{M}$ will be a matroid of rank $r$. When we consider a monomial $h_{F_1}^{a_1} \dotsb h_{F_k}^{a_k}$ in $A^{\bullet}(\mathrm{M})$ or $\underline{A}^{\bullet}(\mathrm{M})$, we always assume the $a_i$ are nonzero, but we allow $k=0$. We do not assume the $F_i$ are distinct unless otherwise stated.

\subsection*{Acknowledgements}
We thank June Huh for helpful conversations about the results in Section~\ref{sec:3}. We thank Luis Ferroni and Vic Reiner for helpful comments on an earlier version of this paper, and we thank Darij Grinberg and the referee for many detailed suggestions. We thank Aldo Conca for explaining the proof of Lemma~\ref{lem:nonzerodivisor} to us. The author is supported by an ARCS fellowship. 

\section{Proof of Theorem~\ref{thm:nonaug} and Theorem~\ref{thm:aug}}\label{sec:2}

\subsection{Straightening monomials}

We begin by using a straightening procedure to prove that the standard monomials, i.e., the elements in Theorem~\ref{SM}, span $A^{\bullet}(\mathrm{M})$. We then use this to prove Theorem~\ref{HR}.

\begin{proposition}\label{prop:SMspanning}
The monomials
\begin{equation*}
\{h_{F_1}^{a_1} \dotsb h_{F_\ell}^{a_\ell} : \emptyset = F_0 < F_1 < \dotsb < F_\ell, \text{ } a_1 \le \operatorname{rk}(F_1), \text{ }  a_i < \operatorname{rk}(F_i) - \operatorname{rk}(F_{i-1}) \text{ for }i=2, \dotsc, \ell\}
\end{equation*}
integrally span ${A}^{\bullet}(\mathrm{M})$.
\end{proposition}
\noindent
We prepare by proving three lemmas. 

\begin{lemma}\label{lemma:chain}
The monomials
\begin{equation*}
\{h_{F_1}^{a_1} \dotsb h_{F_\ell}^{a_\ell} : \emptyset = F_0 < F_1 < \dotsb < F_\ell\}
\end{equation*}
integrally span ${A}^{\bullet}(\mathrm{M})$.
\end{lemma}

\begin{proof}
It suffices to write each monomial of the form $m = h_{G_1}^{b_1} \dotsb h_{G_{\ell}}^{b_{\ell}}$ in terms of the monomials where the flats used form a chain. Suppose that $G_i$ and $G_j$ are distinct and are both maximal in $\{G_1, \dotsc, G_{\ell}\}$. Then we can use the relation
\begin{equation}\label{eq:straighten}
h_{G_i} h_{G_j} = h_{G_i}h_{G_i \vee G_j} + h_{G_j}h_{G_i \vee G_j} - h_{G_i \vee G_j}^2
\end{equation}
to write $m$ as a sum of monomials where there are fewer distinct maximal elements in the set of flats used in each monomial. Repeating this, we can write $m$ as a sum of monomials where, in each monomial, the set of flats used has a unique maximal element. 

We can therefore assume that $G_{\ell}$ is maximal. If $G_i$ and $G_j$ are distinct maximal elements in $\{G_1, \dotsc, G_{\ell - 1}\}$, then we use the relation \eqref{eq:straighten}. As $G_{\ell} \ge G_i \vee G_j$, $G_{\ell}$ will still be maximal in each of the resulting terms. Repeating this argument gives the desired result. 
\end{proof}

\begin{lemma}\label{lemma:dropdown}
Suppose that $G$ covers $F$ in $\overline{\mathcal{L}}_{\mathrm{M}}$, i.e., $F \le G$ and $\operatorname{rk}(G) = \operatorname{rk}(F) + 1$. Then $h_Fh_G = h_G^2$. 
\end{lemma}

\begin{proof}
Because $\mathcal{L}_{\mathrm{M}}$ is atomic, there is an atom $a$ such that $G = F \vee a$. By the defining relations in $A^{\bullet}(\mathrm{M})$, we have that $(h_a - h_G)(h_F - h_G) = 0$ and  $h_ah_F = h_ah_G$. The result follows. 
\end{proof}

\begin{proof}[Proof of Proposition~\ref{prop:SMspanning}]{}
By Lemma~\ref{lemma:chain}, it suffices to show that each monomial $m = h_{F_1}^{a_1} \dotsb h_{F_\ell}^{a_\ell}$, where $\emptyset = F_0 < F_1 < \dotsb < F_\ell$, is either equal to a standard monomial or vanishes. If $m$ is not standard, then either $a_1 > \operatorname{rk}(F_1)$ or $a_i \ge \operatorname{rk}(F_i) - \operatorname{rk}(F_{i-1})$ for some $i \ge 2$. 

Suppose $a_1 > \operatorname{rk}(F_1)$. Choose a chain of covers $\emptyset = G_0 < G_1 < \dotsb < G_{\operatorname{rk}(F_1)} = F_1$. Applying Lemma~\ref{lemma:dropdown} repeatedly, we have that
$$h_{F_1}^{a_1} = h_{G_{\operatorname{rk}(F_1) - 1}}^{a_1 - 1} h_{F_1} = \dotsb = h_{G_1}^{a_1 - \operatorname{rk}(F_1) + 1} h_{G_2} \dotsb h_{F_1}.$$
As $G_1$ is an atom and $a_1 - \operatorname{rk}(F_1) + 1 \ge 2$, we see that $m = 0$ in this case. 

Suppose $a_i \ge \operatorname{rk}(F_i) - \operatorname{rk}(F_{i-1})$ for some $i \ge 2$. Choose a chain of covers $F_{i-1} = G_0 < G_1 < \dotsb < G_{\operatorname{rk}(F_i) - \operatorname{rk}(F_{i-1})} = F_{i}$. Applying Lemma~\ref{lemma:dropdown} repeatedly, we have that
\begin{equation*}\begin{split}
h_{F_{i-1}}^{a_{i-1}} h_{F_i}^{a_i} &=  h_{F_{i-1}}^{a_{i-1}} h_{G_1}^{a_i - \operatorname{rk}(F_i) - \operatorname{rk}(F_{i-1}) + 1} h_{G_2} \dotsb h_{G_{\operatorname{rk}(F_i) - \operatorname{rk}(F_{i-1}) - 1}} h_{F_i} = h_{F_i}^{a_{i-1} + a_{i}}.
\end{split}\end{equation*}
This rewriting decreases the number of flats in the chain. Applying these two operations shows that $m$ is either equal to a standard monomial or vanishes. 
\end{proof}

We say that a multiset $\{F_1, \dotsc, F_r\}$ of flats satisfies the \emph{Hall--Rado} condition if, for all $T \subseteq [r],$ $\operatorname{rk}(\bigvee_{i \in T} F_i) \ge |T|$. 
We say that $T$ \emph{witnesses} the failure of the Hall--Rado condition if $\operatorname{rk}(\bigvee_{i \in T} F_i) < |T|$. 

\begin{proof}[Proof of Theorem~\ref{HR}]
By Proposition~\ref{prop:SMspanning}, $A^r(\mathrm{M})$ is spanned by $h_{E}^r$. Note that $\{E, \dotsc, E\}$ satisfies the Hall--Rado condition, so if $\operatorname{deg}$ is well-defined then it is an isomorphism. We construct $\deg$ by defining a linear map from the degree $r$ part of $\mathbb{Z}[h_F]_{F \in \overline{\mathcal{L}}_{\mathrm{M}}}$ to $\mathbb{Z}$ using the formula in Theorem~\ref{HR} and showing that it descends to $A^{\bullet}(\mathrm{M})$. 
It therefore suffices to prove that if $m = h_{F_1} \dotsb h_{F_{r-2}}$ is a monomial in the degree $r-2$ part of $\mathbb{Z}[h_F]_{F \in \overline{\mathcal{L}}_{\mathrm{M}}}$, then the degree vanishes if we multiply $m$ by any of the defining relations of $A^{\bullet}(\mathrm{M})$. 

\smallskip

We first do the relation $h_a^2 = 0$, for $a$ an atom. Set $F_{r-1} = F_{r} = a$. Then $\{F_1, \dotsc, F_r\}$ does not satisfy the Hall--Rado condition because $1 = \operatorname{rk}(F_{r-1} \vee F_{r}) < 2$. 

\smallskip

We now do the relation $h_ah_F - h_ah_{F \vee a} = 0$, for $a$ an atom and $F \in \overline{\mathcal{L}}_{\mathrm{M}}$. Set $F_{r-1} = a, F_{r} = F,$ and $F_r' = F \vee a$. We need to show that $\{F_1, \dotsc, F_r\}$ satisfies the Hall--Rado condition if and only if $\{F_1, \dotsc, F_{r-1}, F_r'\}$ does. If $\{F_1, \dotsc, F_r\}$ satisfies the Hall--Rado condition, then so does $\{F_1, \dotsc, F_{r-1}, F_r'\}$ because $F_r' \ge F_r$. Suppose $\{F_1, \dotsc, F_r\}$ fails the Hall--Rado condition, so there is some $T \subseteq [r]$ such that $\operatorname{rk}(\bigvee_{i \in T} F_i) < |T|$. We see that $T$ witnesses that $\{F_1, \dotsc, F_r'\}$ also fails the Hall--Rado condition unless $r \in T$ and 
$$\operatorname{rk}(a \vee \bigvee_{i \in T} F_i) = |T|   \text{ and } \operatorname{rk}(a \vee \bigvee_{i \in T} F_i) = \operatorname{rk}(\bigvee_{i \in T} F_i) + 1.$$
In this case, taking $T' = T \cup \{r-1\}$ shows that $\{F_1, \dotsc, F_{r}'\}$ also fails the Hall--Rado condition. 

\smallskip

Finally, we do the relation $(h_{F_{r-1}} - h_{F_{r-1} \vee F_r})(h_{F_{r}} - h_{F_{r-1} \vee F_r}) = 0$, for $F_{r-1}, F_r \in \overline{\mathcal{L}}_{\mathrm{M}}$. Set $S_0 = \{F_1, \dotsc, F_r\}$, $S_1 = \{F_1, \dotsc, F_{r-1}, F_{r-1} \vee F_r\}$, $S_2 = \{F_1, \dotsc, F_{r-1} \vee F_r, F_r\}$, and $S_3 = \{F_1, \dotsc, F_{r-1} \vee F_r, F_{r-1} \vee F_r\}$. If $S_0$ satisfies the Hall--Rado condition, then so do $S_1, S_2$, and $S_3$. Similarly, if $S_1$ or $S_2$ satisfies the Hall--Rado condition, then so does $S_3$.  
There are then two cases which we must prove are impossible.

\textbf{Case 1}: $S_0$ fails the Hall--Rado condition, and $S_1, S_2, S_3$ satisfy the Hall--Rado condition. \\ 
Let $T \subseteq [r]$ witness the failure of the Hall--Rado condition for $S_0$. If $r- 1 \not \in T$, then $T$ witnesses the failure of the Hall--Rado condition for $S_2$. If $r \not \in T$, then $T$ witnesses the failure of the Hall--Rado condition for $S_1$. But if $\{r-1, r\} \subseteq T$, then $T$ witnesses the failure of the Hall--Rado condition for $S_3$. 

\textbf{Case 2}: $S_0, S_1, S_2$ fail the Hall--Rado condition, and $S_3$ satisfies the Hall--Rado condition. \\
Let $T_1 \subseteq [r]$ witness that $S_1$ fails the Hall--Rado condition. We must have $r-1 \in T_1$ and $r \not \in T_1$, as otherwise it would contradict our hypothesis. We can also assume that $T_1 \setminus \{r-1\}$ does not witness the failure of the Hall--Rado condition for $S_3$, so we must have $\operatorname{rk}(\bigvee_{i \in T_1} F_i) = |T_1| - 1$. 
Similarly, we can find $T_2 \subseteq [r]$ with $r \in T_2, r - 1 \not \in T_2$, and $\operatorname{rk}(\bigvee_{i \in T_2 } F_i) = |T_2| - 1$. 
By the submodularity of the rank function, we have that
\begin{equation}{\label{eq:submodular}}
\operatorname{rk}\left ((\bigvee_{i \in T_1} F_i) \wedge (\bigvee_{i \in T_2} F_i)\right ) + \operatorname{rk}\left (\bigvee_{i \in T_1 \cup T_2} F_i\right ) \le \operatorname{rk}\left (\bigvee_{i \in T_1} F_i\right ) + \operatorname{rk}\left (\bigvee_{i \in T_2} F_i \right).
\end{equation}
Set $H = \bigvee_{i \in T_1 \cap T_2} F_i$, so $H \le (\bigvee_{i \in T_1} F_i) \wedge (\bigvee_{i \in T_2} F_i)$.
We may assume that $\operatorname{rk}(H) \ge |T_1 \cap T_2|$, as otherwise $T_1 \cap T_2$ witnesses the failure of the Hall--Rado condition for $S_3$. By \eqref{eq:submodular}, we get that $\operatorname{rk}(\bigvee_{i \in T_1 \cup T_2} F_i) \le |T_1 \cup T_2| -2$. But then $T_1 \cup T_2$ witnesses the failure of the Hall--Rado condition for $S_3$. 
\end{proof}

\subsection{Maps between Chow rings}

For the proof of Theorem~\ref{PD} and \ref{SM}, we will use some maps considered in \cite{BHMPW20a}*{Section 2.6}. 
For a matroid $\mathrm{M}$ of rank $r$ and a flat $G \in \mathcal{L}_{\mathrm{M}}$, let $\mathrm{M}^G$ be the matroid whose lattice of flats is the interval $[\emptyset, G]$, and let $\mathrm{M}_G$ be the matroid whose lattice of flats is the interval $[G, E]$. It is easily checked that these are indeed matroids. 
We will use $G$ to denote the minimal element of $\mathrm{M}_G$ and the maximal element of $\mathrm{M}^G$. 

We say that a flat is \emph{nonempty} if it is not minimal and that it is \emph{proper} if it is not maximal. 
Choose a proper flat $G$ of $\mathrm{M}$. Let $\mathcal{A}$ denote the set of atoms of $\mathrm{M}$ which are not contained in $G$. Set $h_{\emptyset} = 0$, and for a subset $S$ of the atoms of $\mathrm{M}$, we set $\bigvee S = \bigvee_{a \in S} a$. Let
\begin{equation}\label{eq:xgh}
x_G = -\sum_{S \subseteq \mathcal{A}} (-1)^{|S|} h_{G \vee \bigvee S} \in A^{\bullet}(\mathrm{M}).
\end{equation}
Similarly, if $G$ is a proper nonempty flat, we set $x_G = -\sum_{S \subseteq \mathcal{A}} (-1)^{|S|} h_{G \vee \bigvee S} \in \underline{A}^{\bullet}(\mathrm{M}).$ We will always make clear whether we think of $x_G$ as living in $A^{\bullet}(\mathrm{M})$ or $\underline{A}^{\bullet}(\mathrm{M})$.

\begin{lemma}\label{lem:phiG}
Let $G$ be a proper flat of $\mathrm{M}$. There is a surjective ring homomorphism $ \varphi_G \colon A^{\bullet}(\mathrm{M}) \to A^{\bullet}(\mathrm{M}^G) \otimes \underline{A}^{\bullet}(\mathrm{M}_G)$ given by $\varphi_G(h_F) = h_F \otimes 1$ if $F \le G$, and $\varphi_G(h_F) = 1 \otimes h_{F \vee G}$ otherwise. The kernel of $\varphi_G$ is 
$$(h_{F} : F \text{ covers }G) + (h_H - h_{H \vee G} : H \not \le G).$$
\end{lemma}

When $G = \emptyset$, we interpret $A^{\bullet}(\mathrm{M}^{\emptyset})$ as $\mathbb{Z}$, so $\varphi_{\emptyset}$ maps $A^{\bullet}(\mathrm{M})$ to $\underline{A}^{\bullet}(\mathrm{M})$. 

\begin{proof}[Proof of Lemma~\ref{lem:phiG}]
Note that $A^{\bullet}(\mathrm{M})$ is a quotient of $\mathbb{Z}[h_F]_{F \in \overline{\mathcal{L}}_{\mathrm{M}}}$. When we impose the second set of relations in the above ideal, we obtain $\mathbb{Z}[h_F]_{F \in \overline{\mathcal{L}}_{\mathrm{M}}, \, F \le G \text{ or } F>G}$, via the map that sends $h_H$ to $h_{H \vee G}$ if $H \not \le G$. Note that $A^{\bullet}(\mathrm{M}^G) \otimes \underline{A}^{\bullet}(\mathrm{M}_G)$ is a quotient of this ring, and the image of the ideal defining $A^{\bullet}(\mathrm{M})$ is the ideal defining $A^{\bullet}(\mathrm{M}^G) \otimes \underline{A}^{\bullet}(\mathrm{M}_G)$.
\end{proof}

Similarly, we have the following lemma, whose proof is identical to the proof of Lemma~\ref{lem:phiG}.

\begin{lemma}\label{lem:phiGnonaug}
Let $G$ be a proper nonempty flat of $\mathrm{M}$. There is a surjective ring homomorphism $ \varphi_G \colon \underline{A}^{\bullet}(\mathrm{M}) \to \underline{A}^{\bullet}(\mathrm{M}^G) \otimes \underline{A}^{\bullet}(\mathrm{M}_G)$ given by $\varphi_G(h_F) = h_F \otimes 1$ if $F \le G$, and $\varphi_G(h_F) = 1 \otimes h_{F \vee G}$ otherwise. The kernel of $\varphi_G$ is 
$$(h_{F} : F \text{ covers }G) + (h_H - h_{H \vee G} :  H \not \le G).$$
\end{lemma}

Note that, by construction, the following diagram commutes. 
\begin{center}
\begin{tikzcd}
A^{\bullet}(\mathrm{M}) \arrow[r, "\phi_G"] \arrow[d, "\varphi_{\emptyset}"]
& A^{\bullet}(\mathrm{M}^G) \otimes \underline{A}^{\bullet}(M_G) \arrow[d, "\varphi_{\emptyset} \otimes 1"] \\
\underline{A}^{\bullet}(\mathrm{M}) \arrow[r, "\varphi_G"]
& \underline{A}^{\bullet}(\mathrm{M}^G) \otimes \underline{A}^{\bullet}(M_G)
\end{tikzcd}
\end{center}
This allows us to reduce several computations to the augmented case. 
The following lemma will be crucial to our subsequent results. 

\begin{lemma}\label{lem:restrictxaug}
Let $G$ be a proper flat, let $H > G$, and consider $x_H \in A^{\bullet}(\mathrm{M})$. Then $\varphi_G(x_H) = 1 \otimes x_H \in A^{\bullet}(\mathrm{M}^G) \otimes \underline{A}^{\bullet}(\mathrm{M}_G)$. 
\end{lemma}

\begin{proof}
Let $\mathcal{A}$ be the set of atoms of $\mathrm{M}$ not contained in $H$, and let $\mathcal{A}'$ be the set of atoms of $\mathrm{M}_G$ not contained in $H$. There is a surjective map $p \colon \mathcal{A} \to \mathcal{A}'$ given by $a \mapsto G \vee a$. 
Note that for any $T \subset \mathcal{A}$ with $p(T) = S$, we have 
$$1 \otimes h_{H \vee \bigvee S} = \varphi_G(h_{H \vee \bigvee T}).$$
Counting the number of terms $h_{H \vee \bigvee T}$ in the definition \eqref{eq:xgh} of $x_G$ which are mapped to $1 \otimes h_{H \vee \bigvee S}$, it suffices to show that
$$(-1)^{|S|} = \sum_{p(T) = S} (-1)^{|T|}.$$
Let $n_1, \dotsc, n_{|S|}$ be the sizes of the sets $p^{-1}(a)$ for $a \in S$. Note that each $n_i$ is positive.  Then the right-hand side is
\begin{equation*}
((-1 + 1)^{n_1} - 1) \cdot ((-1 + 1)^{n_2} -1) \cdot \ldots \cdot ((-1 + 1)^{n_{|S|}} - 1) = (-1)^{|S|}. \qedhere
\end{equation*}
\end{proof}

The non-augmented version of the previous lemma can be proved in the same way, or it can be deduced by applying $\varphi_{\emptyset}$. 

\begin{lemma}\label{lem:restrictx}
Let $G$ be a proper flat, let $H > G$, and consider $x_H \in \underline{A}^{\bullet}(\mathrm{M})$. Then $\varphi_G(x_H) = 1 \otimes x_H \in \underline{A}^{\bullet}(\mathrm{M}^G) \otimes \underline{A}^{\bullet}(\mathrm{M}_G)$. 
\end{lemma}

\begin{remark}
One can additionally show that, if $H < G$, $\varphi_G(x_H) = x_H \otimes 1$, and that $\varphi_G(x_H) = 0$ if $H$ and $G$ are incomparable. See \cite{BHMPW20a}*{Proposition 2.17}.
\end{remark}

\begin{lemma}\label{lem:cover}
Let $F, G \in \overline{\mathcal{L}}_{\mathrm{M}}$, and suppose that $G$ covers $F \wedge G$. Then $h_Gh_F = h_Gh_{G \vee F}$ in $A^{\bullet}(\mathrm{M})$.
\end{lemma}

\begin{proof}
Because $G$ covers $F \wedge G$, there is an atom $a$ such that $G = (F \wedge G) \vee a$. Then $G \vee F = F \vee a$, so Lemma~\ref{lemma:dropdown} gives
$$h_{F}h_{G \vee F} = h_{F} h_{F \vee a} = h_{F \vee a}^2 = h_{G \vee F}^2. $$
The result follows by using that $h_Fh_{G \vee F} = h_{G \vee F}^2$ in the equation
\begin{equation*}
h_Gh_F = h_Gh_{G \vee F} + h_Fh_{G \vee F} - h_{G \vee F}^2.\qedhere
\end{equation*}
\end{proof}

\begin{lemma}\label{lem:annihilator}
Let $G$ be a proper flat. The kernel of $\varphi_G \colon A^{\bullet}(\mathrm{M}) \to A^{\bullet}(\mathrm{M}^G) \otimes \underline{A}^{\bullet}(\mathrm{M}_G)$ is contained in the annihilator $\operatorname{ann}(x_G)$. Similarly, if $G$ is a proper nonempty flat, the kernel of $\varphi_G  \colon \underline{A}^{\bullet}(\mathrm{M}) \to \underline{A}^{\bullet}(\mathrm{M}^G) \otimes \underline{A}^{\bullet}(\mathrm{M}_G)$ is contained in $\operatorname{ann}(x_G)$.
\end{lemma}

\begin{proof}
We do the augmented case; the non-augmented case can be proved similarly or deduced by applying $\varphi_{\emptyset}$. We first show that $x_G \cdot h_F = 0$ if $F$ covers $G$, i.e., $F = G \vee a$ for some atom $a \not \in G$. Let $\mathcal{A}$ be the set of atoms not contained in $G$. Then
\begin{equation*}
x_G \cdot h_{G \vee a} = - \sum_{S \subseteq \mathcal{A}} (-1)^{|S|} h_{G \vee \bigvee S}h_{G \vee a}.
\end{equation*}
Suppose that $S$ does not contain $a$. 
If $(G \vee \bigvee S) \wedge (G \vee a) \not= G \vee a$, then $(G \vee \bigvee S) \wedge (G \vee a) = G$, so it is covered by $G \vee a$. In this case, Lemma~\ref{lem:cover} then gives that $h_{G \vee \bigvee S}h_{G \vee a} = h_{G \vee a \vee \bigvee S} h_{G \vee a}$. We see that, for any $S$ not containing $a$, we have
$$h_{G \vee \bigvee S}h_{G \vee a} = h_{G \vee \bigvee (S \cup a)}h_{G \vee a}.$$
Therefore the terms in the sum indexed by those $S$ which contain $a$ cancel with the terms indexed by those $S$ which do not contain $a$, and so the sum is $0$.

We now show that $x_G(h_H - h_{H \vee G}) = 0$ if $H \not \le G$. That is, we need to show that
\begin{equation}\label{eq:NTS}
x_G(h_H - h_{H \vee G}) = - \sum_{S \subseteq \mathcal{A}} (-1)^{|S|} h_{G \vee \bigvee S}h_{H} + \sum_{S \subseteq \mathcal{A}} (-1)^{|S|} h_{G \vee \bigvee S}h_{H \vee G}
\end{equation}
vanishes.
For $S \subseteq \mathcal{A}$, the relation $(h_F - h_{F \vee K})(h_K - h_{F \vee K}) = 0$ in $A^{\bullet}(\mathrm{M})$ implies that
$$h_{G \vee \bigvee S}h_H - h_{G \vee \bigvee S}h_{H \vee G} =  h_{H}h_{H \vee G \vee \bigvee S} - h_{H \vee G}h_{H \vee G \vee \bigvee S}$$
Substituting this into the right-hand side of \eqref{eq:NTS}, we get that
$$x_G (h_H - h_K) = -\sum_{S \subseteq \mathcal{A}}(-1)^{|S|} h_H h_{H \vee G \vee \bigvee S} + \sum_{S \subseteq \mathcal{A}}(-1)^{|S|}h_{H \vee G} h_{H \vee G \vee \bigvee S}.$$
Because $H \not \le G$, we may choose an atom $a \le H$ with $a \not \le G$. Then the terms in the first sum indexed by those $S$ which contain $a$ cancel with the terms indexed by those $S$ that do not contain $a$, so the first sum is $0$. Similarly, the second sum is $0$. 
\end{proof}

In particular, the map $A^{\bullet}(\mathrm{M}) \to A^{\bullet}(\mathrm{M})/\operatorname{ann}(x_G)$ factors through $\varphi_G$, and similarly in the non-augmented setting. This will be a useful aid in computations. 

\subsection{Projection formulas and dragon-Hall--Rado}
We now show that the maps constructed in the previous section are compatible with degree maps. Along the way, we prove Theorem~\ref{dHR}. First we prove that the standard monomials span $\underline{A}^{\bullet}(\mathrm{M})$.

\begin{proposition}\label{prop:SMspanningnonaug}
The monomials
\begin{equation*}
\{h_{F_1}^{a_1} \dotsb h_{F_\ell}^{a_\ell} : \emptyset = F_0 < F_1 < \dotsb < F_\ell, \text{ }  a_i < \operatorname{rk}(F_i) - \operatorname{rk}(F_{i-1}) \text{ for }i=1, \dotsc, \ell\}
\end{equation*}
integrally span $\underline{A}^{\bullet}(\mathrm{M})$.
\end{proposition}

\begin{proof}
The map $\varphi_{\emptyset} \colon A^{\bullet}(\mathrm{M}) \to \underline{A}^{\bullet}(\mathrm{M})$ is surjective, so by Proposition~\ref{prop:SMspanning}, it suffices to show that $h_F^{\operatorname{rk}(F)} = 0$ in $\underline{A}^{\bullet}(\mathrm{M})$. By Lemma~\ref{lemma:dropdown}, $h_F^{\operatorname{rk}(F)}$ is divisible by $h_a$ for any atom $a$ contained in $F$, and so it is $0$ in $\underline{A}^{\bullet}(\mathrm{M})$. 
\end{proof}

We now prove Theorem~\ref{dHR}. We deduce it from Theorem~\ref{HR}, although one can also argue analogously to the proof of Theorem~\ref{HR}. 

\begin{proof}[Proof of Theorem~\ref{dHR}]
By Proposition~\ref{prop:SMspanningnonaug}, $\underline{A}^{r-1}(\mathrm{M})$ is spanned by $h_E^{r-1}$, so if the degree map is well-defined then it is an isomorphism. 
By Lemma~\ref{lem:annihilator}, there is a surjective ring homomorphism $\underline{A}^{\bullet}(\mathrm{M}) \to A^{\bullet}(\mathrm{M})/\operatorname{ann}(x_{\emptyset})$. Note that $A^{\bullet}(\mathrm{M})/\operatorname{ann}(x_{\emptyset})$ is identified with the ideal $(x_{\emptyset})$, with degree shifted by $1$. We define the degree map via the composition
$$\deg \colon \underline{A}^{r-1}(\mathrm{M}) \to A^{r-1}(\mathrm{M})/\operatorname{ann}(x_{\emptyset}) \to A^r(\mathrm{M}) \to \mathbb{Z},$$
where the second map is multiplication by $x_{\emptyset}$. Let $\mathcal{A}$ be the set of atoms of $\mathrm{M}$. In order to prove Theorem~\ref{dHR}, it suffices to show that
\begin{equation}
-\sum_{S \subseteq \mathcal{A}} (-1)^{|S|} \deg (h_{\bigvee S} h_{F_1} \dotsb h_{F_{r-1}})  = \begin{cases} 1 & \text{ for all }\emptyset \not= T \subseteq [r-1], \text{ }\operatorname{rk}(\bigvee_{i \in T} F_i) \ge |T| + 1 \\ 0 & \text{ otherwise},\end{cases} \label{eq:sum}
\end{equation}
as the left-hand side is, by definition, $\deg(x_{\emptyset} h_{F_1} \dotsb h_{F_{r-1}})$.
Suppose that $\{F_1, \dotsc, F_{r-1}\}$ satisfies the dragon-Hall--Rado condition. If $S$ is nonempty, $\{ F_1, \dotsc, F_{r-1},\bigvee S\}$ satisfies the Hall--Rado condition. We see that every term in the sum in \eqref{eq:sum} is $1$ except for $S = \emptyset$, so the sum is $1$. 

Now suppose that a multiset $\{F_1, \dotsc, F_{r-1}\}$ fails the dragon-Hall--Rado condition. Let 
$$\mathcal{S} = \{S \subseteq \mathcal{A} : \{F_1, \dotsc, F_{r-1}, \bigvee S\} \text{ fails the Hall--Rado condition}\}.$$ Clearly $\mathcal{S}$ is downward closed: if $T \subseteq S \in \mathcal{S}$, then $T \in \mathcal{S}$. Because $\{F_1, \dotsc, F_{r-1}\}$ fails the dragon-Hall--Rado condition, there is some $i$ such that $\{a \in \mathcal{A} : a \le F_i\}$ is contained in $\mathcal{S}$. 

Let $I_1, I_2 \in \mathcal{S}$. 
We claim that $I_1 \cup I_2 \in \mathcal{S}$. If there is a witness to the failure of the Hall--Rado condition for $F_1, \dotsc, F_{r-1}, \bigvee I_1$ which does not contain $\bigvee I_1$, then this is immediate. Otherwise, choose $T_1, T_2 \subseteq [r-1]$ such that $\{F_j : j \in T_1\} \cup \{\bigvee I_1\}$  and $\{F_j : j \in T_2\} \cup \{\bigvee I_2\}$ witness the failure of the Hall--Rado condition, so 
$$\operatorname{rk}(\bigvee I_1 \vee \bigvee_{j \in T_1} F_j) \le |T_1|,$$
and similarly for $T_2$. By the monotonicity of the rank function, we have that
$$\operatorname{rk}((\bigvee I_1 \vee \bigvee_{j \in T_1} F_j) \wedge (\bigvee I_2 \vee \bigvee_{j \in T_2} F_j)) \ge \operatorname{rk}(\bigvee_{j \in T_1 \cap T_2} F_j) \ge |T_1 \cap T_2|,$$
where the last inequality is by the assumption that no witness to the failure of the Hall--Rado condition is contained in $[r-1]$. By the submodularity of the rank function function, we have that
$$|T_1 \cap T_2| + \operatorname{rk}(\bigvee (I_1 \cup U_2) \vee \bigvee_{j \in T_1 \cup T_2} F_j) \le \operatorname{rk}(\bigvee I_1 \vee \bigvee_{j \in T_1} F_j) + \operatorname{rk}(\bigvee I_2 \vee \bigvee_{j \in T_2} F_j).$$
This implies that $\operatorname{rk}(\bigvee (I_1 \cup I_2)\vee \bigvee_{j \in T_1 \cup T_2} F_j) < |T_1 \cup T_2| + 1$, so $I_1 \cup I_2 \in \mathcal{S}$, as desired. 

Therefore $\mathcal{S}$ contains a maximal element, so it is a Boolean lattice of size at least $2$. We see that the sum in \eqref{eq:sum} is zero. 
\end{proof}

Let $G$ be a proper flat of $\mathrm{M}$. The tensor product of the degree maps gives an isomorphism $\deg \colon A^{\operatorname{rk}(G)}(\mathrm{M}^G) \otimes \underline{A}^{r-1 - \operatorname{rk}(G)}(\mathrm{M}_G) \to \mathbb{Z}$. If $G$ is nonempty, there is an isomorphism $\deg \colon \underline{A}^{\operatorname{rk}(G) - 1}(\mathrm{M}^G) \otimes \underline{A}^{r-1 - \operatorname{rk}(G)}(\mathrm{M}_G) \to \mathbb{Z}$. It will be convenient to extend the degree maps by zero to the entirety of $A^{\bullet}(\mathrm{M})$, $\underline{A}^{\bullet}(\mathrm{M})$ and so on. 
The following lemmas will be critical to the proof of Theorem~\ref{PD} and \ref{PDnonaug}.

\begin{lemma}\label{prop:projection}
Let $y \in A^{r-1}(\mathrm{M})$, and let $G$ be a proper flat. Then 
$$\deg(\varphi_G(y)) = \deg(x_G \cdot y).$$
\end{lemma}

\begin{proof}
Lemma~\ref{lem:annihilator} implies that the right-hand side only depends on $\varphi_G(y)$. As the degree $r-1$ part of $A^{\bullet}(\mathrm{M}^G) \otimes \underline{A}^{\bullet}(\mathrm{M}_G)$ is $\mathbb{Z}$, the maps $y \mapsto \deg(\varphi_G(y))$ and $y \mapsto \deg(x_G \cdot y)$ are equal up to a constant. 

Let $y = h_{G}^{\operatorname{rk}(G)} h_E^{r - 1 - \operatorname{rk}(G)}$. We see from Theorem~\ref{HR} and Theorem~\ref{dHR} that $\deg(\varphi_G(y)) = 1$. Let $\mathcal{A}$ be the set of atoms of $\mathrm{M}$ not contained in $G$. We have
$$\deg(x_G \cdot y) = - \sum_{S \subseteq \mathcal{A}} (-1)^{|S|} \deg(h_{G \vee \bigvee S} h_{G}^{\operatorname{rk}(G)} h_E^{r - 1 - \operatorname{rk}(G)}).$$
The term $S = \emptyset$ vanishes because it does not satisfy the Hall--Rado condition; all other terms are $1$, so the sum is $1$. 
\end{proof}

\begin{lemma}\label{prop:projectionnonaug}
Let $y \in \underline{A}^{r-2}(\mathrm{M})$, and let $G$ be a proper nonempty flat. Then 
$$\deg(\varphi_G(y)) = \deg(x_G \cdot y).$$
\end{lemma}

\begin{proof}
This can be proved as in the proof of Lemma~\ref{prop:projection}. Alternatively, we can choose a lift $\tilde{y} \in A^{\bullet}(\mathrm{M})$ such that $\varphi_{\emptyset}(\tilde{y}) = y$ and apply Lemma~\ref{prop:projection} twice to $x_{\emptyset} \cdot x_G \cdot \tilde{y}$. 
\end{proof}

\subsection{Poincar\'{e} duality and linear independence}

Now that we have access to Lemma~\ref{prop:projection} and Lemma~\ref{prop:projectionnonaug}, we can begin our proof of Theorem~\ref{PD}. Our strategy is closely related to \cite{BES}*{Proposition 3.3.10}, which is based on \cite{Ham17}*{Proposition 3.2}.
Let $m = h_{F_{i_1}}^{a_1} \dotsb h_{F_{i_k}}^{a_k}$ be a standard monomial for $A^{\bullet}(\mathrm{M})$, with $\operatorname{rk}(F_{i_j}) = i_j$. Extend the chain $F_{i_1} < \dotsb < F_{i_k}$ to a maximal chain of flats $\emptyset = F_0 < F_1 < \dotsb < F_{r} = E$. Let $\mathcal{G}_m$ be the collection of flats obtained by removing from this chain the $a_j$ flats immediately below $F_{i_j}$ for each $j$ and removing $E$. Because $m$ is a standard monomial, $\{F_{i_1}, \dotsc, F_{i_k}\} \setminus \{E\} \subseteq \mathcal{G}_m$. We do this process and obtain a collection of flats $\mathcal{G}_m$ for each standard monomial $m$. We call $\mathcal{G}_m$ the \emph{essential flats} of $m$.

We will now prove the key propositions that allow us to prove Theorem~\ref{SM} and Theorem~\ref{PD}. See Example~\ref{ex:PDexample} for an example illustrating their proofs. 

\begin{proposition}\label{prop:deg1}
Let $m \in A^{\ell}(\mathrm{M})$ be a standard monomial, and let $\mathcal{G}_m = \{G_1 <  \dotsb < G_k\}$ be the essential flats. Then $\deg(m \cdot x_{G_1} \dotsb x_{G_k}) = 1$. 
\end{proposition}

\begin{proof}
Set $G_0 = \emptyset$ and set $G_{k+1} = E$. Note that possibly $G_1 = \emptyset$ as well. 
Applying Lemma~\ref{prop:projection} and Lemma~\ref{prop:projectionnonaug} repeatedly, using Lemma~\ref{lem:restrictxaug} and Lemma~\ref{lem:restrictx}, we can write the degree as a degree in $A^{\bullet}(\mathrm{M}^{G_1}_{G_0}) \otimes \underline{A}^{\bullet}(\mathrm{M}_{G_1}^{G_2}) \otimes  \dotsb  \otimes \underline{A}^{\bullet}(\mathrm{M}_{G_k}^{G_{k+1}})$. Here if $\operatorname{rk}(G_{i+1}) = \operatorname{rk}(G_{i}) + 1$, then we interpret $\underline{A}^{\bullet}(\mathrm{M}^{G_{i+1}}_{G_i})$ as $\mathbb{Z}$, and similarly if $\operatorname{rk}(G_1) = 0$.  

The only terms in the tensor product which are not $\mathbb{Z}$ are $\underline{A}^{\bullet}(\mathrm{M}_{G_i}^{G_{i+1}})$ if  $\operatorname{rk}(G_{i+1}) - \operatorname{rk}(G_i) > 1$ and $A^{\bullet}(\mathrm{M}_{G_0}^{G_1})$ if $\operatorname{rk}(G_1) > 0$. 
From the construction of $\mathcal{G}_m$, we see that, if $i \not= 0$, then $h_{G_{i+1}}^{\operatorname{rk}(G_{i+1}) - \operatorname{rk}(G_i) - 1}$ appears in $m$. If $i = 0$ and $\operatorname{rk}(G_1) > 0$, then $h_{G_{1}}^{\operatorname{rk}(G_{1})}$ appears in $m$. 
In the first case, after applying $\varphi_{G_i}$ for all $G_i \in \mathcal{G}_m$, $h_{G_{i+1}}^{\operatorname{rk}(G_{i+1}) - \operatorname{rk}(G_i) - 1}$ lands in top degree in $\underline{A}^{\bullet}(\mathrm{M}_{G_i}^{G_{i+1}})$. In the second case, $h_{G_1}^{\operatorname{rk}(G_1)}$ lands in top degree in $A^{\bullet}(\mathrm{M}^{G_1}_{G_0})$. By Theorem~\ref{HR} and ~\ref{dHR}, we see that the degree is $1$. 
\end{proof}

For a standard monomial $m = h_{F_1}^{a_1} \dotsb h_{F_k}^{a_k}$, we set $\delta(m) = (\sum_{\operatorname{rk}(F_i) \le 1} a_i, \sum_{\operatorname{rk}(F_i) \le 2} a_i, \dotsc, \sum_{\operatorname{rk}(F_i) \le r} a_i)$. 

\begin{proposition}\label{prop:degvanishing}
Let $m \in A^{\ell}(\mathrm{M})$ be a standard monomial, and let $\mathcal{G}_m = \{G_1 <  \dotsb < G_k\}$ be the essential flats. Let $m' \in A^{\ell}(\mathrm{M})$ be a standard monomial which has $\deg(m' \cdot \prod_{G \in \mathcal{G}_m} x_G) \not= 0$. Then either $m = m'$ or $\delta(m') > \delta(m)$ lexicographically. 
\end{proposition}

\begin{proof}
Set $G_0 = \emptyset$ and set $G_{k+1} = E$. 
As in the proof of Proposition~\ref{prop:deg1}, we can write the degree as a degree in $A^{\bullet}(\mathrm{M}^{G_1}_{G_0}) \otimes \underline{A}^{\bullet}(\mathrm{M}_{G_1}^{G_2}) \otimes  \dotsb  \otimes \underline{A}^{\bullet}(\mathrm{M}_{G_k}^{G_{k+1}})$. As before, the top degree of $\underline{A}^{\bullet}(\mathrm{M}_{G_i}^{G_{i+1}})$ is $\operatorname{rk}(G_{i+1}) - \operatorname{rk}(G_i) - 1$. Also, the top degree of $A^{\bullet}(\mathrm{M}_{G_0}^{G_1})$ is $\operatorname{rk}(G_1)$. 

Let $m' = h_{F_1}^{a_1} \dotsb h_{F_{\ell}}^{a_{\ell}}$. Let $G_j$ be the least element of $\mathcal{G}_m$ with $G_j \ge F$. After applying $\varphi_{G}$ for all $G \in \mathcal{G}_m$, $h_{F_i}^{a_i}$ is mapped to $1 \otimes \dotsb \otimes h_{G_{j-1} \vee F_i}^{a_i} \otimes \dotsb \otimes 1$. In particular,  for each $i > 0$ with $\operatorname{rk}(G_{i+1}) - \operatorname{rk}(G_i) > 1$, $\deg(m' \cdot \prod_{G \in \mathcal{G}_m} x_G)$ vanishes unless there are flats $F_j, \dotsc, F_p$ appearing in $m'$ with $a_j + \dotsb + a_p = \operatorname{rk}(G_{i+1}) - \operatorname{rk}(G_i) - 1$ and $F_q \le G_{i+1}$, $F_q \not \le G_i$ for each $q=j, \dotsc, p$. Similarly, if $\deg(m' \cdot \prod_{G \in \mathcal{G}_m} x_G)$ is nonzero and $\operatorname{rk}(G_1) > 0$, then there must be $F_1, \dotsc, F_p$ appearing in $m'$ with $a_1 + \dotsb + a_p = \operatorname{rk}(G_1)$ and $F_q \le G_1$ for each $q=1, \dotsc, p$. 
Adding these conditions up, this implies that the degree vanishes if $\delta(m') < \delta(m)$ or if $\delta(m) = \delta(m')$ and $m \not= m'$. 
\end{proof}

\begin{example}\label{ex:PDexample}
Let $\mathrm{M}$ be the Boolean matroid of rank $6$, i.e., $\mathcal{L}_{\mathrm{M}}$ is the Boolean lattice on $6$ elements. Let $F_i = \{1, \dotsc, i\}$ for $i = 0, 1, \dotsc, 6$. Let $m = h_{F_2}h_{F_5}^2$, so the essential flats $\mathcal{G}_m$ are $\{F_0, F_2, F_5\}$. We apply $\varphi_{F_0}$, then $\varphi_{F_2}$, and then $\varphi_{F_5}$ to write $\operatorname{deg}(m \cdot x_{F_0} x_{F_2}x_{F_5})$ as a degree in 
$$\underline{A}^{\bullet}(\mathrm{M}^{F_2}) \otimes \underline{A}^{\bullet}(\mathrm{M}^{F_5}_{F_2}) \otimes \underline{A}^{\bullet}(\mathrm{M}_{F_5}) = \underline{A}^{\bullet}(\mathrm{M}^{F_2}) \otimes \underline{A}^{\bullet}(\mathrm{M}^{F_5}_{F_2}).$$
We have $\varphi_{F_5} \circ \varphi_{F_2} \circ \varphi_{F_0}(h_{F_2}) = h_{F_2} \otimes 1$ and $\varphi_{F_5} \circ \varphi_{F_2} \circ \varphi_{F_0}(h_{F_5}^2) = 1 \otimes h_{F_5}^2$, so the degree is $1$. 

Let $m'$ be a standard monomial where the rank of the smallest flat appearing is at least $3$, so $\delta(m') < \delta(m)$ lexicographically. Then, for each $h_G$ appearing in $m'$, we have
 $$ \varphi_{F_5} \circ \varphi_{F_2} \circ \varphi_{F_0}(h_{G}) = \begin{cases} 1 \otimes h_{G \vee F_2} & G \le F_5 \\ 0 & G \not \le F_5, \end{cases}$$  In particular, no term appearing in $m'$ maps to something of the form $h_F \otimes 1$. This implies that $\deg(m' \cdot x_{F_0}x_{F_2}x_{F_5}) = 0$. 
\end{example}

\begin{proof}[Proof of Theorem~\ref{PD} and ~\ref{SM}]
Fix $0 \le k \le r$. Choose a total order $<$ on the set of standard monomials of degree $k$ such that $m < m'$ if $\delta(m) < \delta(m')$ lexicographically. For each standard monomial $m$, we have an element $d(m) \coloneqq \prod_{G \in \mathcal{G}_m} x_G \in A^{r - k}(\mathrm{M})$. By Proposition~\ref{prop:deg1} and Proposition~\ref{prop:degvanishing}, the matrix whose rows and columns are labeled by standard monomials of degree $k$, and whose entry indexed by $(m, m')$ is $\operatorname{deg}(m \cdot d(m'))$, is lower triangular with $1$'s on the diagonal. This implies that the standard monomials of degree $k$ are linearly independent, so, by Proposition~\ref{prop:SMspanning}, they are a basis. 

We also see that $\operatorname{rank} A^k(\mathrm{M}) \le \operatorname{rank} A^{r-k}(\mathrm{M})$. Replacing $k$ by $r-k$, we see that $\operatorname{rank} A^k(\mathrm{M}) = \operatorname{rank} A^{r-k}(\mathrm{M})$, and so the $d(m)$ rationally span $A^{r-k}(\mathrm{M}) \otimes \mathbb{Q}$. Because the determinant of the pairing between $A^k(\mathrm{M})$ and the subgroup of $A^{r-k}(\mathrm{M})$ spanned by the $d(m)$ is $1$, we see that the $d(m)$ must integrally span $A^{r-k}(\mathrm{M})$, which proves Theorem~\ref{PD}. 
\end{proof}

In order to prove Theorem~\ref{PDnonaug} and ~\ref{SMnonaug}, we will need an analogue of Proposition~\ref{prop:deg1} and ~\ref{prop:degvanishing} for non-augmented Chow rings. We will deduce these from their augmented versions. 

For a standard monomial $m = h_{F_{i_1}}^{a_1} \dotsb h_{F_{i_k}}^{a_k}$ for $\underline{A}^{\bullet}(\mathrm{M})$, we define $\mathcal{G}_m$ in the same way as in the augmented setting: extend the chain $F_{i_1} < \dotsb < F_{i_k}$ to a maximal chain of flats $\emptyset = F_0 < F_1 < \dotsb < F_{r} = E$. Let $\mathcal{G}_m$ be collection of flats obtained by removing from this chain the $a_j$ flats below $F_{i_j}$ for each $j$ and removing $E$. Because $m$ is a standard monomial for $\underline{A}^{\bullet}(\mathrm{M})$, $\emptyset \in \mathcal{G}_m$. We define $\delta(m)$ in the same way as for standard monomials for $A^{\bullet}(\mathrm{M})$. 

\begin{proposition}\label{prop:deguppertriang}
Let $m$ be a standard monomial of $\underline{A}^{\bullet}(\mathrm{M})$. Then 
\begin{enumerate}
\item we have that $\deg(m \cdot \prod_{\emptyset \not= G \in \mathcal{G}_m} x_G) = 1$.
\item for each standard monomial $m'$ for $\underline{A}^{\bullet}(\mathrm{M})$ with $\deg(m' \cdot \prod_{\emptyset \not= G \in \mathcal{G}_m} x_G) \not= 0$, either $m = m'$ or $\delta(m') > \delta(m)$ lexicographically. 
\end{enumerate}
\end{proposition}

\begin{proof}
Let $m = h_{F_1}^{a_1} \dotsb h_{F_k}^{a_k}$.
By Proposition~\ref{prop:projection}, we have that the degree $\deg(m \cdot \prod_{\emptyset \not= G \in \mathcal{G}_m} x_G)$ in $\underline{A}^{\bullet}(\mathrm{M})$ is equal to the degree in $A^{\bullet}(\mathrm{M})$ of $h_{F_1}^{a_1} \dotsb h_{F_k}^{a_k}$ times $\prod_{G \in \mathcal{G}_m} x_G$. The result then follows from Proposition~\ref{prop:deg1} and Proposition~\ref{prop:degvanishing}.
\end{proof}

\begin{proof}[Proof of Theorem~\ref{PDnonaug} and ~\ref{SMnonaug}]
Fix $0 \le k \le r$. Choose a total order $<$ on the set of standard monomials of degree $k$ such that $m < m'$ if $\delta(m) < \delta(m')$ lexicographically. For each standard monomial $m$, we have an element $d(m) \coloneqq \prod_{\emptyset \not= G \in \mathcal{G}_m} x_G \in \underline{A}^{r - 1 -  k}(\mathrm{M})$. By Proposition~\ref{prop:deguppertriang}, the matrix whose rows and columns are labeled by standard monomials of degree $k$, and whose entry indexed by $(m, m')$ is $\operatorname{deg}(m \cdot d(m'))$, is lower triangular with $1$'s on the diagonal. As in the proof of Theorem~\ref{PD} and \ref{SM}, this implies the linear independence of the standard monomials and Poincar\'{e} duality. 
\end{proof}

\section{Gradings by $\mathcal{L}_{\mathrm{M}}$}\label{sec:3}
One corollary of our approach is the existence of a ``grading'' of $A^{\bullet}(\mathrm{M})$ by $\mathcal{L}_{\mathrm{M}}$, which we now study. Special cases of this decomposition were used in \cite{EHKR}*{Section 5.1} and \cite{Rains}. 
For a flat $F$, let $A^{\bullet}(\mathrm{M})_F$ be the span of the monomials $h_{G_1}^{a_1} \dotsb h_{G_k}^{a_k}$, where $G_1 \vee \dotsb \vee G_k = F$. For example, $A^{\bullet}(\mathrm{M})_{\emptyset} = \operatorname{span}(1)$.

\begin{proposition}\label{prop:grading}
We have a direct sum decomposition 
$$A^{\bullet}(\mathrm{M}) = \bigoplus_{F \in \mathcal{L}_{\mathrm{M}}} A^{\bullet}(\mathrm{M})_F.$$
\end{proposition}

\begin{proof}
There is clearly such a decomposition for $\mathbb{Z}[h_F]_{F \in \overline{\mathcal{L}}_{\mathrm{M}}}$, and the relations in $A^{\bullet}(\mathrm{M})$ respect this decomposition. 
\end{proof}

\begin{lemma}\label{lem:contraction}
Let $F$ be a proper nonempty flat of $\mathrm{M}$. There is a graded ring isomorphism $\bigoplus_{G \le F} A^{\bullet}(\mathrm{M}^F)_G \stackrel{\sim}{\to} \bigoplus_{G \le F} A^{\bullet}(\mathrm{M})_G$ given by $h_G \mapsto h_G$. In particular, the graded abelian groups $A^{\bullet}(\mathrm{M}^F)_F$ and $A^{\bullet}(\mathrm{M})_F$ are isomorphic.
\end{lemma}

\begin{proof}
Note that the subring $\bigoplus_{G \le F} A^{\bullet}(\mathrm{M})_G$ of $A^{\bullet}(\mathrm{M})$ is generated by $h_G$ for $G \le F$, and the relations in $\bigoplus_{G \le F} A^{\bullet}(\mathrm{M}^F)_F$ and $\bigoplus_{G \le F} A^{\bullet}(\mathrm{M})_F$ are the same. 
\end{proof}

If $\operatorname{rk}(\mathrm{M}) > 0$, the \emph{truncation} $\operatorname{Tr}\mathrm{M}$ is the matroid whose lattice of flats $\mathcal{L}_{\operatorname{Tr}\mathrm{M}}$ is obtained by removing the flats $F$ with $\operatorname{rk}(F) = \operatorname{rk}(E) - 1$. There is a surjective ring homomorphism $A^{\bullet}(\mathrm{M}) \to A^{\bullet}(\operatorname{Tr}\mathrm{M})$ given by $h_F \mapsto h_{E}$ if $\operatorname{rk}(F) = \operatorname{rk}(E) - 1$ and $h_F \mapsto h_F$ otherwise. The kernel of this map is $(h_E - h_F : \operatorname{rk}(F) = \operatorname{rk}(E) - 1)$. 

\begin{lemma}\label{lem:truncation}
Let $\operatorname{M}$ be a matroid of rank $r > 0$. 
There is an isomorphism of graded $A^{\bullet}(\mathrm{M})$-modules $A^{\bullet}(\operatorname{Tr}\mathrm{M})[-1] \stackrel{\sim}{\to} A^{\bullet}(\mathrm{M})_{E}$ given by $1 \mapsto h_{E}$. 
\end{lemma}

\begin{proof}
By Lemma~\ref{lemma:chain} and its proof, $A^{\bullet}(\mathrm{M})_{E}$ is the ideal generated by $h_E$. We have an identification of $A^{\bullet}(\mathrm{M})$-modules $A^{\bullet}(\mathrm{M})/\operatorname{ann}(h_E)[-1] \stackrel{\sim}{\to} (h_E)$ given by multiplication by $h_E$.

We claim that the kernel of the map $A^{\bullet}(\mathrm{M}) \to  A^{\bullet}(\operatorname{Tr}\mathrm{M})$ is $\operatorname{ann}(h_E)$, which concludes the proof. If $F$ is a flat of $\mathrm{M}$ with $\operatorname{rk}(F) = r- 1$, then $h_E(h_F - h_E) = 0$ by Lemma~\ref{lemma:dropdown}, so $\operatorname{ann}(h_E)$ contains the kernel. Note that $h_E^{r-1}$ is nonzero in $A^{\bullet}(\mathrm{M})/\operatorname{ann}(h_E)$ by Theorem~\ref{HR}. 
Poincar\'{e} duality for $A^{\bullet}(\operatorname{Tr}\mathrm{M})$ then implies that the surjective map $A^{\bullet}(\operatorname{Tr}\mathrm{M}) \to A^{\bullet}(\mathrm{M})/\operatorname{ann}(h_E)$ is an isomorphism because it is an isomorphism in degree $r-1$. Indeed, Poincar\'{e} duality implies that every nonzero ideal of $A^{\bullet}(\operatorname{Tr} \mathrm{M})$ intersects $A^{r-1}(\operatorname{Tr} \mathrm{M})$ nontrivially. 
\end{proof}

Combining Lemma~\ref{lem:contraction} with Lemma~\ref{lem:truncation} gives that, if $F$ is a nonempty flat, then $A^{\bullet}(\mathrm{M})_F \stackrel{\sim}{\to} A^{\bullet}(\operatorname{Tr}\mathrm{M}^F)[-1]$ as graded abelian groups. In particular, $A^{\bullet}(\mathrm{M})_F$ vanishes above degree $\operatorname{rk}(F)$ and is $1$-dimensional in degree $\operatorname{rk}(F)$. By Theorem~\ref{SM}, we see that $A^{\operatorname{rk}(F)}(\mathrm{M})_F$ is spanned by $h_F^{\operatorname{rk}(F)}$. In particular, a monomial $h_{G_1}^{a_1} \dotsb h_{G_k}^{a_{k}}$, with $a_1 + \dotsb + a_{k} = \operatorname{rk}(F)$ and $G_i \le F$ for each $i$ is either $0$ or equal to $h_F^{\operatorname{rk}(F)}$. 

The \emph{graded M\"{o}bius algebra} $H^{\bullet}(\mathrm{M})$ of a matroid $\mathrm{M}$ is a ring which is $\bigoplus_{F \in \mathcal{L}_{\mathrm{M}}} y_F \cdot \mathbb{Z}$ as an abelian group, with multiplication $y_F \cdot y_G = y_{F \vee G}$ if $\operatorname{rk}(F) + \operatorname{rk}(G) = \operatorname{rk}(F \vee G)$ and $y_F \cdot y_G = 0$ otherwise. Note that $H^{\bullet}(\mathrm{M})$ is graded, with $y_F$ in degree $\operatorname{rk}(F)$, and that $H^{\bullet}(\mathrm{M})$ is generated in degree $1$. 
A detailed study of modules over the graded M\"{o}bius algebra is central to the proof of the top-heavy conjecture in \cite{BHMPW20b}. One of the key results is the following realization of $H^{\bullet}(\mathrm{M})$ as a subring of $A^{\bullet}(\mathrm{M})$. We give a simple proof. 

\begin{proposition}\cite{BHMPW20a}*{Proposition 2.15}
There is an injective ring homomorphism $H^{\bullet}(\mathrm{M}) \to A^{\bullet}(\mathrm{M})$, defined by sending $y_a$ to $h_a$ for each atom of $\mathcal{L}_{\mathrm{M}}$.
\end{proposition}

\begin{proof}
For a flat $F$, let $a_1, \dotsc, a_{\operatorname{rk}(F)}$ be atoms with $\bigvee_{i=1}^{\operatorname{rk}(F)} a_i = F$. By Theorem~\ref{HR}, $\operatorname{deg}(h_{a_1} \dotsb h_{a_{\operatorname{rk}(F)}} h_{E}^{r - \operatorname{rk}(F)}) = 1$. In particular, by the discussion above, we have that $h_{a_1} \dotsb h_{a_{\operatorname{rk}(F)}} = h_F^{\operatorname{rk}(F)}$.
By the direct sum decomposition in Proposition~\ref{prop:grading}, the subalgebra generated by the $h_a$ for $a$ an atom has a basis given by $\{h_F^{\operatorname{rk}(F)}\}_{F \in \mathcal{L}_{\mathrm{M}}}$. 
We therefore see that this algebra is isomorphic to $H^{\bullet}(\mathrm{M})$. 
\end{proof}

For a matroid $\mathrm{M}$, let $\mathrm{H}_{\mathrm{M}}(t)$ be the Hilbert series of $A^{\bullet}(\mathrm{M})$, and let $\underline{\mathrm{H}}_{\mathrm{M}}(t)$ be the Hilbert series of $\underline{A}^{\bullet}(\mathrm{M})$. These polynomials, which are sometimes called (augmented) Chow polynomials, have been extensively studied in \cite{MotivicZetaMatroid} and especially \cite{FMSV}, where the authors derive several recursive relations between them. 
The analysis in this section immediately generalizes to $\underline{A}^{\bullet}(\mathrm{M})$, and this gives new recursions for $\mathrm{H}_{\mathrm{M}}(t)$ and $\underline{\mathrm{H}}_{\mathrm{M}}(t)$. 

\begin{corollary}
We have that
$$\mathrm{H}_{\mathrm{M}}(t) = 1 + \sum_{F \in \overline{\mathcal{L}}_{\mathrm{M}}} t \cdot \mathrm{H}_{\operatorname{Tr} \mathrm{M}^F}(t) \quad \text{and} \quad \underline{\mathrm{H}}_{\mathrm{M}}(t) = 1 + \sum_{F \in \mathcal{L}_{\mathrm{M}}, \,\operatorname{rk}(F) \ge 2} t \cdot \underline{\mathrm{H}}_{\operatorname{Tr} \mathrm{M}^F}(t).$$
\end{corollary}

Using Lemma~\ref{lem:truncation}, we give a second proof of Theorem~\ref{SM}; Theorem~\ref{SMnonaug} can be proved similarly. Note that the proof of Lemma~\ref{lem:truncation} used Poincar\'{e} duality for $A^{\bullet}(\operatorname{Tr} \mathrm{M})$. 

\begin{proof}[Proof of Theorem~\ref{SM}]
We have the decomposition 
$$A^{\bullet}(\mathrm{M}) = \mathbb{Z} \oplus \bigoplus_{F \in \overline{\mathcal{L}}_{\mathrm{M}}} A^{\bullet}(\operatorname{Tr} \mathrm{M}^F)[-1].$$
By induction, we have a standard monomial basis for each summand on the right-hand side. In the above decomposition, a monomial $h_{G_1}^{a_1} \dotsb h_{G_k}^{a_k}$ in $A^{\bullet}(\operatorname{Tr} \mathrm{M}^F)$ is mapped to the monomial $h_{G_1}^{a_1} \dotsb h_{G_k}^{a_k}\cdot h_F$ in $A^{\bullet}(\mathrm{M})$. As $h_{G_1}^{a_1} \dotsb h_{G_k}^{a_k}$ is standard in $A^{\bullet}(\operatorname{Tr} \mathrm{M}^F)$ if and only if $h_{G_1}^{a_1} \dotsb h_{G_k}^{a_k}\cdot h_F$ is standard in $A^{\bullet}(\mathrm{M})$, this implies the result. 
\end{proof}


\begin{remark}
The geometry of the decomposition in Proposition~\ref{prop:grading} is explained in \cite{Rains}*{Section 2}. For each $F \in \overline{\mathcal{L}}_{\mathrm{M}}$, there is an idempotent projection $A^{\bullet}(\mathrm{M}) \to A^{\bullet}(\mathrm{M})$ given by $h_G \mapsto h_G$ if $G \le F$, and otherwise $h_G \mapsto 0$. This map factors through $A^{\bullet}(\mathrm{M}^F)$, and, when $\mathrm{M}$ is realizable, it arises from a retraction of the augmented wonderful variety of a realization whose image is the augmented wonderful variety of a realization of $\mathrm{M}^F$. These projections commute, and $A^{\bullet}(\mathrm{M})_F$ is the set of elements of $A^{\bullet}(\mathrm{M})$ which are fixed by the projection associated to $F$ and killed by the projection associated to $G$ for all $G < F$. 
\end{remark}

\section{Algebras with straightening laws}\label{ref:sec4}

In this section, we construct an algebra with straightening law which is closely related to the (augmented) Chow ring of a matroid. Algebras with straightening laws, also known as \emph{ordinal Hodge algebras} \cite{HodgeAlgebras}, are certain algebras which are equipped with a standard monomial basis. We follow \cite{BrunsVetter} for conventions on algebras with straightening laws.

\begin{definition}\label{def:ASL}
Let $B^{\bullet}$ be a graded algebra over a ring $R$, and let $(\Pi, \le)$ be a finite poset equipped with an injection $\Pi \to B^{\bullet}$ which identifies $\Pi$ with a subset of $B^{\bullet}$. Assume that $B^0 = R$, and that the elements of $\Pi$ are homogeneous of positive degree. We say that $B^{\bullet}$ is an \emph{algebra with straightening law} over $\Pi$ if
\begin{enumerate}
\item\label{ASL1} the \emph{standard  monomials} $\{y_1^{a_1} \dotsb y_{k}^{a_k} : y_1 \le \dotsb \le y_k \in \Pi\}$ form an $R$-basis for $B^{\bullet}$, and
\item \label{ASL2} for each $x, y \in \Pi$ incomparable, when we express $xy$ in terms of the standard monomial basis $xy = \sum a_{\mu} \mu$, where $b_{\mu} \in R$ and $\mu$ is a standard monomial, each $\mu$ with $b_{\mu} \not= 0$ contains a factor of some $z \in \Pi$ with $z < x$ and  $z < y$. 
\end{enumerate}
\end{definition}

We will work in a more general setting than matroids. 
Let $\mathcal{L}$ be a finite \emph{meet-semilattice}, i.e., a finite partially ordered set where any two elements $x, y$ have a greatest lower bound $x \wedge y$. There is a minimal element $\hat{0}$ of $\mathcal{L}$. Our main example will be $\overline{\mathcal{L}}_{\mathrm{M}}^{\mathrm{op}}$, i.e., the inverted poset of flats of a matroid with the empty set removed. Here the minimal element is $E$. 

\begin{theorem}\label{thm:ASL}
Let $\mathcal{L}$ be a finite meet-semilattice, and let 
$$B^{\bullet}(\mathcal{L}) = \frac{\mathbb{Z}[h_x]_{x \in \mathcal{L}}}{((h_x - h_{x \wedge y})(h_y - h_{x \wedge y}) : x, y \in \mathcal{L})}, \text{ with $h_x$ in degree $1$}.$$
Then $B^{\bullet}(\mathcal{L})$ is an algebra with straightening law over $\mathcal{L}$. 
\end{theorem}

When $\mathcal{L} = \overline{\mathcal{L}}_{\mathrm{M}}^{\mathrm{op}}$, then 
$$B^{\bullet}(\mathcal{L}) = \frac{\mathbb{Z}[h_F]_{F \in \overline{\mathcal{L}}_{\mathrm{M}}}}{((h_{F} - h_{G \vee F})(h_G - h_{G \vee F}) : F, G \in \overline{\mathcal{L}}_{\mathrm{M}})}.$$
There is a quotient map from $B^{\bullet}(\mathcal{L})$ to $A^{\bullet}(\mathrm{M})$ and $\underline{A}^{\bullet}(\mathrm{M})$. 
In particular, the straightening procedure used in the proof of Lemma~\ref{lemma:chain} is a shadow of the straightening law on $B^{\bullet}(\mathcal{L})$. This is made precise in the proof of Theorem~\ref{SMnonaug} at the end of this section.

The order complex of $\mathcal{L}$ is the simplicial complex whose faces are given by chains in $\mathcal{L}$. Let $C^{\bullet}(\mathcal{L})$ denote the Stanley--Reisner ring of the order complex of $\mathcal{L}$, with variables $\{s_x : x \in \mathcal{L}\}$. 
The theory of algebras with straightening laws shows that $B^{\bullet}(\mathcal{L})$ has a Gr\"{o}bner degeneration to $C^{\bullet}(\mathcal{L})$. 
Note that $B^{\bullet}(\mathcal{L})$ is itself isomorphic to the Stanley--Reisner ring of the order complex of $\mathcal{L}$, via the map which sends $h_x$ to $\sum_{y \le x} s_y$. Note that this is not an isomorphism of algebras with straightening laws when $C^{\bullet}(\mathcal{L})$ is considered with the injection $\mathcal{L} \to C^{\bullet}(\mathcal{L})$ by $x \mapsto s_x$. 


The proof of Theorem~\ref{thm:ASL} is similar to the geometric argument used to show that the homogeneous coordinate ring of a Schubert variety in the Grassmannian is an algebra with straightening law, see \cite{HodgeAlgebras}*{Proposition 1.3}.  We prepare for the proof of Theorem~\ref{thm:ASL} with a lemma. We thank Aldo Conca for explaining the proof to us. 

\begin{lemma}\label{lem:nonzerodivisor}
The element $h_{\hat{0}}$ is not a zero-divisor in $B^{\bullet}(\mathcal{L})$. 
\end{lemma}

\begin{proof}
Choose an ordering $x_1, \dotsc, x_n$ of the elements of $\mathcal{L}$ where $x_n = \hat{0}$. For $i=1, \dotsc, n-1$, set $u_i = h_{x_i} - h_{x_{i+1}}$. Then the elements $u_1, \dotsc, u_{n-1}, h_{x_n}$ form a basis for the degree $1$ part of the polynomial ring $\mathbb{Z}[h_F]_{F \in \hat{\mathcal{L}}}$. After we change to this basis, none of the elements of ideal defining $B^{\bullet}(\mathcal{L})$ involve $h_{x_n} = h_{\hat{0}}$. As the ideal of $B^{\bullet}(\mathcal{L})$ is not the unit ideal because it is graded, $h_{\hat{0}}$ is not a zero-divisor. 
\end{proof}

\begin{proof}[Proof of Theorem~\ref{thm:ASL}]
If $x, y \in \mathcal{L}$ are incomparable, then the relation 
$$h_xh_y = h_xh_{x \wedge y} + h_yh_{x \wedge y} - h_{x \wedge y}^2$$
shows that Definition~\ref{def:ASL}(\ref{ASL2}) is satisfied. 
The argument in Lemma~\ref{lemma:chain} shows that $B^{\bullet}(\mathcal{L})$ is spanned by standard monomials, so it suffices to show that the standard monomials are linearly independent. 
Adjoin a maximal element $\hat{1}$ to $\mathcal{L}$ to form $\hat{\mathcal{L}}$. Let $B^{\bullet}(\mathcal{L})_x$ be the span of monomials $h_{y_1}^{a_1} \dotsb h_{y_k}^{a_k}$ such that $y_1 \wedge \dotsb \wedge y_k = x$. For example, $B^{\bullet}(\mathcal{L})_{\hat{1}} = \operatorname{span}(1)$. 
As in the proof of Proposition~\ref{prop:grading}, there is a direct sum decomposition
$$B^{\bullet}(\mathcal{L}) = \bigoplus_{x \in \hat{\mathcal{L}}} B^{\bullet}(\mathcal{L})_x.$$
It therefore suffices to show that the standard monomials $h_{x_1}^{a_1} \dotsb h_{x_k}^{a_k}$, with $x_1 \le \dotsb \le x_k$, are linearly independent in $B^{\bullet}(\mathcal{L})_{x_1}$. 

Let $\hat{\mathcal{L}}_{x}$ be the interval $[x, \hat{1}]$ in $\hat{\mathcal{L}}$. We see that $\hat{\mathcal{L}}_x \setminus \hat{1}$ is a meet semilattice, and, as in the proof of Lemma~\ref{lem:contraction}, we have
$$B^{\bullet}(\mathcal{L})_x \stackrel{\sim}{\to} B^{\bullet}(\hat{\mathcal{L}}_x \setminus \hat{1})_{\hat{0}} \text{ as abelian groups}.$$
In particular, by induction it suffices to show that the standard monomials where $h_{\hat{0}}$ appears are linearly independent. If there was a linear dependence among the standard monomials where $h_{\hat{0}}$ appears, then that would imply that $h_{\hat{0}}$ is a zero-divisor, which contradicts Lemma~\ref{lem:nonzerodivisor}. 
\end{proof}

One could alternatively establish the linear independence of the standard monomials using the isomorphism $C^{\bullet}(\mathcal{L}) \to B^{\bullet}(\mathcal{L})$. 

Using Theorem~\ref{thm:ASL} in the case $\mathcal{L} = \overline{\mathcal{L}}_{\mathrm{M}}^{\mathrm{op}}$, we can give another proof of Theorem~\ref{SMnonaug}. One can prove Theorem~\ref{SM} using a similar but more lengthy argument. 

\begin{proof}[Proof of Theorem~\ref{SMnonaug}]
We will use Theorem~\ref{thm:ASL} to construct a linear endomorphism of $\psi$ of $\mathbb{Z}[h_F]_{F \in \overline{\mathcal{L}}_{\mathrm{M}}}$ whose image is the span of the standard monomials and whose kernel is the ideal defining $\underline{A}^{\bullet}(\mathrm{M})$. This gives an (abelian group) direct sum decomposition of the polynomial ring, which implies that the standard monomials form a basis for $\underline{A}^{\bullet}(\mathrm{M})$. 

Let $C \subseteq \mathbb{Z}[h_F]_{F \in \overline{\mathcal{L}}_{\mathrm{M}}}$ be the linear span of $\{h_{F_1}^{a_1} \dotsb h_{F_\ell}^{a_\ell} : \emptyset < F_1 < \dotsb < F_\ell\}$. Theorem~\ref{thm:ASL} gives a surjective linear map $\psi_1 \colon \mathbb{Z}[h_F]_{F \in \overline{\mathcal{L}}_{\mathrm{M}}} \to C$: we consider the image of an element of $\mathbb{Z}[h_F]_{F \in \overline{\mathcal{L}}_{\mathrm{M}}}$ inside of $B^{\bullet}(\overline{\mathcal{L}}_{\mathrm{M}}^{\mathrm{op}})$ and then express it in terms of the standard monomial basis there. 

The proofs of Proposition~\ref{prop:SMspanning} and Proposition~\ref{prop:SMspanningnonaug} give a map $\psi_2$ from $C$ to the linear span of the standard monomials for $\underline{A}^{\bullet}(\mathrm{M})$: a monomial $h_{F_1}^{a_1} \dotsb h_{F_\ell}^{a_\ell}$ with $\emptyset < F_1 < \dotsb < F_{\ell}$ is either $0$ in $\underline{A}^{\bullet}(\mathrm{M})$ or is equal to a particular standard monomial for $\underline{A}^{\bullet}(\mathrm{M})$ (which is independent of the choices involved). We define $\psi_2$ to be the map which sends $h_{F_1}^{a_1} \dotsb h_{F_\ell}^{a_\ell}$ to either $0$ or this standard monomial. We define $\psi$ to be $\psi_2 \circ \psi_1$. 

In proof of Proposition~\ref{prop:SMspanning} and Proposition~\ref{prop:SMspanningnonaug}, a procedure is described which writes any monomial in $\underline{A}^{\bullet}(\mathrm{M})$ in terms of the standard monomials: use the relations of the form $(h_F - h_{F \vee G})(h_G - h_{F \vee G}) = 0$ to write the monomial as a sum of monomials corresponding to chains, and each monomial corresponding to a chain is either $0$ or equal to a standard monomial. 
The well-definedness of $\psi$ implies that this procedure is independent of the choices involved.

It is clear that $\psi$ surjects onto the span of the standard monomials; we need to show that the ideal defining $\underline{A}^{\bullet}(\mathrm{M})$ is in the kernel of $\psi$. By construction, the kernel of $\psi$ is contained in the ideal defining $\underline{A}^{\bullet}(\mathrm{M})$. As $\psi$ is linear, it suffices to prove that $\psi$ kills the product of any monomial $m$ with a generator of the ideal defining $\underline{A}^{\bullet}(\mathrm{M})$. 

By construction, $\psi_1(m \cdot (h_F - h_{F \vee G})(h_G - h_{F \vee G})) = 0$ for any incomparable flats $F, G$, so $\psi$ kills $m \cdot (h_F - h_{F \vee G})(h_G - h_{F \vee G})$ as well. 

We need to check that, for any atom $a$ and monomial $m$, we have $\psi(m \cdot h_a) = 0$. We apply the procedure used to compute $\psi_1$ (as described in Lemma~\ref{lemma:chain}) to $m \cdot h_a$, i.e., we find a pair of flats $\{F, G\}$, where $h_F$ and $h_G$ appear in $m \cdot h_a$, which are incomparable and which are maximal with these properties. We then use the relation $h_Fh_G = h_Fh_{F \vee G} + h_G h_{F \vee G} - h_{F \vee G}^2$. If $a \not \in \{F, G\}$, then all resulting terms are divisible by $h_a$. Note that applying $\psi_2$ kills any term where $h_a$ appear. 

It therefore suffices to understand the case when there is a flat $F$ such that $h_F$ appears in $m \cdot h_a$, $F$ is incomparable with $a$, and for all $G$ with $h_G$ appearing $m \cdot h_a$, either $G = a$, $G \le F$, or $G \ge F \vee a$. Define $m'$ by $m = m' \cdot h_F$
We use the relation 
$$m \cdot h_a = m' \cdot h_ah_F = m' \cdot h_ah_{F \vee a} + m' \cdot h_{F}h_{F \vee a} - m' \cdot h_{F \vee a}^2.$$
The terms appearing after further straightening of $m' \cdot h_{F}h_{F \vee a}$ will be the same as those in  $m' \cdot h_{F \vee a}^2$, except with $h_{F} h_{F \vee a}$ replaced by $h_{F \vee a}^2$. But these terms will cancel when we apply $\psi_2$.  
\end{proof}

\bibliographystyle{alpha}
\bibliography{matroid.bib}

\end{document}